\documentclass[12pt,a4paper]{article}
\usepackage{amsmath,amssymb,amsthm,amsfonts,amscd,euscript,verbatim, t1enc, newlfont, mathtools}
\usepackage{hyperref}
\usepackage{yfonts}
\usepackage{graphicx}
\usepackage[all]{xy}
\theoremstyle{definition}
\newtheorem{theo}{Theorem}[section]

\newtheorem{pr}[theo]{Proposition}

 \newtheorem{coro}[theo]{Corollary}

		 \newtheorem{ass}[theo]{Assumption}
\theoremstyle{remark}
\newtheorem{rema}[theo]{Remark}

\theoremstyle{definition}
\newtheorem{defi}[theo]{Definition}

\newcommand \cu{\underline{C}}

\newcommand\ilim\varinjlim

\newcommand\chow{\operatorname{Chow}}

\newcommand\dmgm{\operatorname{DM}_{gm}}
\newcommand\dmgmop{\operatorname{DM}_{gm}^{op}}
\newcommand\smc{\operatorname{SmCor}}
\newcommand\mg{M_{gm}}
\newcommand\mgc{M_{gm}^c}

\newcommand\dms{\operatorname{DM}_{sv}}
\newcommand\dmsp{\operatorname{DM}'_{sv}}
\newcommand\dmsc{\operatorname{DM}^c_{sv}}

\newcommand\motn{\operatorname{Mot_{num}}} 
\newcommand\tnum{{t_{num}}} 

\newcommand\hetl{H^{et}_{\ql}{}}
\newcommand\rhetl{{RH}^{et}_{\ql}{}}

\newcommand\kalg{k^{alg}}
\newcommand\spe{\operatorname{Spec}}
\newcommand\var{\operatorname{Var}}

\newcommand\sv{\operatorname{SmVar}}
\newcommand\spv{\operatorname{SmPrVar}}
\newcommand\sm{\underline{\operatorname{SmVar}}}
\newcommand\af{\mathbb{A}}
\newcommand\pt{\operatorname{pt}}
\newcommand\lan{\langle}
\newcommand\ra{\rangle}
\newcommand\ns{\{0\}}

\newcommand\q{{\mathbb{Q}}}
\newcommand\obj{\operatorname{Obj}}

\DeclareMathOperator\cha{\operatorname{char}}
\newcommand\vecto{\operatorname{vect}}

\newcommand\ql{{\mathbb{Q}_l}}

\newcommand\z{{\mathbb{Z}}}
\newcommand\com{\mathbb{C}}

\newcommand\p{\mathbb{P}}

\DeclareMathOperator\co{\operatorname{Cone}}

\begin{document}

 \title{Conservativity of realizations implies that numerical motives are Kimura-finite and motivic zeta functions are rational}
 \author{M.V. Bondarko
   \thanks{The main results of the paper were  obtained under support of the Russian Science Foundation grant no. 16-11-10200.}
	}
 \maketitle
\begin{abstract}
 
In 
 this note we prove the following: if the (\'etale or de Rham) realization functor is conservative on the category $\dmgm$ of $\q$-linear Voevodsky motives then  motivic zeta functions of arbitrary varieties are rational 
 and numerical motives are Kimura-finite. 
The latter statement immediately implies that the category $\motn$ of numerical motives is (essentially) 
 Tannakian.

This observation becomes actual due to the recent announcement of J. Ayoub that the 
de Rham cohomology realization is conservative on $\dmgm(k)$ whenever $\cha k=0$. We apply this statement to  exterior powers of motives coming from generic hyperplane sections of smooth affine varieties. 



\end{abstract}

\tableofcontents

 \section*{Introduction}

In this note 
we (essentially) prove that the conservativity of a Weil realization for the category $\dmgm$ of Voevodsky geometric motives with rational coefficients over an infinite perfect field $k$ implies certain nice "finite-dimensionality" properties of motives. The actuality of these statements comes from the recent text \cite{ayoubcon} where the conservativity of the restriction of the De Rham realization $RH_{dr}$ 
 to Chow motives is established under the assumption that $p=\cha k=0$. Moreover, in \cite{bwcomp} certain statements that allow to deduce
the conservativity of $RH_{dr}$  on the category $\dmgm\supset\chow$ from Theorem II of \cite{ayoubcon} were proved; so, J. Ayoub plans to prove the latter conservativity statement eventually. Even though the current version of the proof of loc. cit. contains a gap, it appears to be interesting to know the consequences of this conservativity of realizations assertion; note that it is equivalent to the conservativity of the \'etale realization  (combined with the base change functor from $k$ to $\kalg$; see  Remark \ref{rkim}(2) below) that 
will be denoted by $\rhetl$ 
 (here $l$ is an arbitrary prime; this equivalence is given by the easy Proposition \ref{pcons} below). 
 For this reason the author chose to formulate our main 
 assumption as follows.  


\begin{ass}\label{assmain}
There exists $l\neq p=\cha k$ such that the functor $\rhetl$ is conservative, i.e., an object $M$ of $\dmgm$ is zero whenever $\rhetl(M)=0$.
\end{ass}


Recall that this statement is conjecturally valid for 
  arbitrary base fields (and all $l\in \p\setminus\{p\}$). 

Now, the starting point of this paper is that for any smooth affine $k$-variety $A$ and any its 
 generic hyperplane section $Z$ (corresponding to any  embedding of $A$ into a projective space) 
 the relative \'etale cohomology for $(Z,A)$ is concentrated in the degree $\dim A$ only (see Corollary 3.4.1 of \cite{katzaff}; this statement can also be deduced from a statement in \cite[\S3.3.1]{be87}). 
Combining  this statement with Assumption \ref{assmain} we obtain that the 
 motif $\co(\mg(Z\to A))$ is Kimura finite dimensional (either evenly or oddly, depending on the parity of $\dim A$). 

The latter statement easily implies two nice consequences (even though these Kimura-finite motives are not pure, i.e., are not shifts of Chow motives; cf. Remark \ref{rnori}(3)). We prove that the Kapranov's motivic zeta function of any $k$-variety $X$ is rational; here we take motivic zeta functions of varieties and motives that 
belong to 
 $K^0(\chow)[[t]]$ (after \S2 of \cite{gul}; the idea to relate Kimura-finiteness to motivic zetas originates from \cite{andsurvey}).  This statement it very easy to understand when $X$ is smooth projective (and see \S\ref{sgs} for the general case); however, its proof 
 heavily relies on Voevodsky motives and the isomorphism $K_0(\dmgm)\cong K_0(\chow)$ (as established in \cite{mymot} and \cite{bzp}; cf. Remark \ref{rnori}(4)  below).

We also prove that our conservativity Assumption \ref{assmain} implies that all numerical motives are Kimura-finite; the proof uses the exact "motivic" weight complex functor (as introduced in \cite{mymot}). Thus the category $\motn$ is Tannakian. 

Note also that the aforementioned generic hyperplane section argument is not the only one that gives Kimura-finite objects of $\dmgm$. An alternative method is the vanishing cycle one as  described in \cite{ayoubvan}. Its advantage is that it allows producing finite dimensional motives over  characteristic $p>0$ fields from that over characteristic $0$ ones.
  We will probably apply an argument of this sort in the next version of this preprint to reduce the main statements of this paper for all values of $p$ to Assumption \ref{assmain} for the case $p=0$; see Remarks \ref{rnori}(2) and \ref{rinfield}(1) below for some more detail. However, the author doubts that this alternative method is "strictly better" than the one that we use in the current text.

Now let us describe the contents of the paper. 

In \S\ref{smot} we briefly recall some basics on motives, their \'etale cohomology, and Kimura-finiteness. 

In \S\ref{sfd} we prove that Assumption \ref{assmain} 
implies 
the finite dimensionality  of motives 
of the form $\co(\mg(Z\to A))$ as above 
 (in  $\dmgm$).

In \S\ref{szeta} we study motivic zeta functions and prove that 
 Assumption \ref{assmain} implies that they are rational for motives belonging to a "large" dense subcategory $\dms$ of $\dmgm$ (that contains all motives of smooth varieties).

In \S\ref{skim} we prove (under Assumption \ref{assmain}) that all numerical motives are finite-dimensional; hence the category of numerical motives is essentially Tannakian.

In \S\ref{suppl} we prove that 
	 the rationality of zeta functions statement can be extended from $\dms$ to an (a priori) larger subcategory of $\dmgm$ that contains motives with compact support of arbitrary varieties; thus we obtain the rationality of 
	 motivic zeta functions of varieties. We also recall the 
 well-known relation between the conservativity of different realizations in the case $p=0$; so we relate our Assumption \ref{assmain} to the aforementioned claim of Ayoub. 
	
	Most of definitions mentioned in this paper are well-known; since we do not need much detail on them, we prefer not to include some of them in the text (and give references instead; see the survey paper \cite{andsurvey}). We also make several remarks 
	  concerning literature on motives, their Kimura-finiteness, and related matters. 

The author is deeply grateful to  prof. J. Ayoub for interesting discussions concerning the conservativity of realizations and its applications. The comments of prof. B. Kahn were  very helpful as well.
Moreover, the author is extremely thankful to the officers of the Max Planck Institut f\"ur Mathematik for the wonderful working conditions during the writing of this text.


\section{Some 
 preliminaries}\label{smot}

First we introduce some notation.

In this paper  $k$ will denote a perfect ground field; $p=\cha k$ (it can be zero). We will also assume that $k$ is infinite (see Remark \ref{rinfield}(2) below).

It will be convenient for us to use the term "$k$-variety" for reduced  separated  (not necessarily integral) schemes of finite type over $\spe k$; we will write $\var$ for the set of all $k$-varieties. 
 Accordingly,  the set of smooth varieties (resp. of smooth projective varieties) over $k$ will be denoted by $\sv$ (resp. by $\spv$), and we do not assume these schemes to be connected. We will write $\sm$ for the corresponding category of smooth varieties.


 $\pt$ is the point $\spe k$, $\af^n$ is the $n$-dimensional
affine space (over $k$); $\p^1$ is the projective line. 

In this paper all motives will be $\q$-linear ones (so, we will omit $\q$ in the notation); by default, they will be $k$-ones.

Respectively, we will write $\dmgm$ for the triangulated category of geometric Voevodsky motives over $k$ with rational coefficients; see 
 \cite[\S5.3]{kellyast},   \cite[Appendix A.2]{kellyth},  \cite[\S2.1,\ 4.3]{1}, and (Proposition 1.3.3 of) \cite{bokum}.

The following properties of Voevodsky motives are well-known;  for this reason we will not give precise references to them.

So, $\dmgm$ is a 
  small tensor triangulated category equipped with a functor $\mg:\sm\to \dmgm$. Moreover, $\q=\mg(\pt)$ is the unit object 
	 of $\dmgm$, and  the morphisms $\pt\to \p^1\to \pt$ give a decomposition $ \mg(\p^1)\cong \q\bigoplus \q\lan 1 \ra$; here $\q\lan 1 \ra$ is a certain Lefschetz motif that is $\otimes$-invertible in $\dmgm$. 
	 We will write  $-\lan j \ra$ for the $j$th iteration of the endofunctor $-\otimes\q\lan 1\ra:\dmgm\to \dmgm$  for any $j\in \z$ (in \cite{1} this functor was denoted by $-(j)[2j]$); recall that $\dmgm$ coincides with its smallest dense subcategory that contains $\cup_{j\in \z}\mg(\sv)\lan j \ra$. Lastly, $\dmgm$ is {\it Karoubian}, i.e., any idempotent endomorphism $p:M\to M$ in $\dmgm$ gives a direct sum decomposition of $M$ in it; thus $p$ has a (categorical) image.

Since the "main" motivic categories $\dmgm$, $\chow$, and $\motn$ of this paper are idempotent complete by definition, the word "retract" below is a synonym of "direct summand".

For a field $K$ we will write $K-\vecto$ for the category of finite-dimensional vector spaces.

We will say that a set $B$ of objects of $\dmgm$ {\it strongly generates} a subcategory $\cu\subset \dmgm$ if $\cu$ is the  smallest strictly full triangulated subcategory of $\dmgm$ containing $B$.




Next we define Kimura-finite motives. 

\begin{defi}\label{dkim}

Let $M$ be an object of $\dmgm$.

1. For $m>0$ we will write $\wedge^{m}M$ (resp. $\operatorname{Sym}^m(M)$) for the categorical image of the idempotent endomorphism $\sum_{\sigma\in S_m} (-1)^{i(s\sigma)}\sigma_{*M^{\otimes m}}: M^{\otimes m}\to M^{\otimes m}$ (resp., of $\sum_{\sigma\in S_m} \sigma_{*M^{\otimes m}}/m!$; we permute the components of the tensor power here).

2.  $M$ will be said to be {\it Kimura-even} if for some $r>0$ we have $\wedge^{r}M=0$, i.e., if $\sum_{\sigma\in S_r} (-1)^{i(s\sigma)}\sigma_{*M^{\otimes r}}=0$.


3. $M$ will be said to be {\it Kimura-odd} if for some $r>0$ we have  $\operatorname{Sym}^r(M)=0$, i.e.,  if $\sum_{\sigma\in S_r} \sigma_{*M^{\otimes r}}=0$.

4. We will say that $M$ is {\it Kimura finite dimensional} (or just Kimura-finite or finite dimensional)  if it can be presented as a $\dmgm$-direct sum $M_+\bigoplus M_-$, where $M_+$ is Kimura-even and $M_-$ is Kimura-odd.

\end{defi}

\begin{rema}\label{rkim}
1. Kimura-finite objects have several nice properties. However, instead of recalling them just 
now we will only relate the finite-dimensionality of motives to their cohomology.

2. So we recall that for any prime $l$ distinct from $p$ there exists a   tensor exact $\ql$-\'etale realization functor $\rhetl:
\dmgmop\to D^b(\ql-\vecto)$ 
whose composition with the functor $\mg$ gives the functor of total $\ql$-\'etale cohomology  
($X\mapsto \rhetl(X_{k^{alg}})$, i.e., this composition sends a variety $X$ into the cohomology of its base change to the algebraic closure of $k$); see Theorem 4.3 of \cite{ivorretale}.\footnote{The reader may certainly assume that $k$ is algebraically closed itself; cf. Remark \ref{rinfield}(2) below.}
\end{rema}


\section{On Kimura-finiteness of certain mixed motives}\label{sfd}

The key statement of this paper is as follows.

\begin{theo}\label{tfdim}
Assume that Assumption \ref{assmain} is fulfilled for $\rhetl$ (and our $k$). Then the following statements are valid.

1. Let $M\in \obj \dmgm$. If the cohomology 
$\ql$-vector spaces $\hetl^s(M)$ of $\rhetl(M)$ are zero in even (resp. odd) degrees then $M$ is Kimura-odd (resp., Kimura-even).  

2. Let $A$ be a smooth connected affine $k$-variety; choose an embedding of $A$ into an affine space. Then 
  a generic hyperplane section $Z$ of $A$ is smooth, 
 and the motives $\co(\mg(Z\to A))\lan i \ra$\footnote{Recall that the aforementioned functor $\mg:\sm\to \dmgm$ factors through an exact functor $K^b(\smc)\to \dmgm$; see \cite{1}. Thus the motif $\co(\mg(Z\to A))$ can also be described as the image of the corresponding two-term complex $\dots\to 0\to Z\to A\to 0\dots \in \obj K^b(\smc)$; here we put $A$ in degree $0$.} are Kimura-even for all $i\in \z$ whenever $\dim A$ is even and are Kimura-odd in the opposite 
case.

\end{theo}
\begin{proof}
1. The standard definition of the tensor product on $D^b(\ql-\vecto)$ implies that for $r=1+\sum_{s\in \z}\dim_{\ql}(\hetl^s(M))$ we have
$\operatorname{Sym}^r(\rhetl(M))=0$ (resp. $\wedge^{r}{\rhetl(M)}=0$); cf. Proposition 3.9 of \cite{kim}. Since $\rhetl$ is a tensor functor, it remains to apply Assumption \ref{assmain} to obtain the result.

2. Certainly, a generic choice of $Z$ is smooth. 

According to assertion 1, it remains to  verify that $H^s(\rhetl(\co(\mg(Z\to A)))=\ns$ unless $s=\dim A$ (for a "generic" $Z$; note that the twist $-\lan i \ra$ essentially shifts cohomology by $2i$). Since both $Z$ and $A$ are affine, 
applying the Artin Vanishing of \'etale cohomology we obtain that the latter statement is  equivalent to the bijectivity of the corresponding morphisms $h^s:\hetl^s(A)\to \hetl^s(Z)$ for $s<\dim A-1$ along with the injectivity of $h^{\dim A-1}$.
The latter statement follows immediately from Corollary 3.4.1(2) of \cite{katzaff} (just take the constant sheaf $\overline{\q}_l$ 
 for $\mathfrak{F}$ and $f=0$ in it).




\end{proof}

\begin{rema}\label{rnori}
1. The "Affine Weak Lefschetz" results of \cite[\S3.4]{katzaff} appear to be closely related to the Basic Nori Lemma (see \S2.5 of \cite{hubnori}) and so also to \S3.3.1 of \cite{be87}.

2. An alternative source of finite-dimensional motives is given by motivic vanishing cycles (cf. \cite{ayoubvan}). After some work, it can 
probably be used to obtain "enough finite-dimensional motives" in the case $p>0$ from Assumption \ref{assmain} for characteristic $0$ fields.\footnote{The author is deeply grateful to prof. J. Ayoub for approving this idea and for sending him a sketch of the corresponding argument. 
 This reasoning will probably be added to a succeeding version of this paper.} So the author hopes to 
 deduce the $p>0$-case of (the conclusion of) Theorem \ref{tnum} below together with a certain version of Theorem \ref{tzeta}  from Assumption \ref{assmain} for the case $p=0$ (and so, from the 
 recent results of Ayoub; see Remark \ref{rinfield}(1)  below).


3. The widely believed to be true Conjecture 7.1 of \cite{kim} predicts that (Chow) motives of smooth projective varieties are finite dimensional. Applying Lemma 3.7(3) of \cite{andsurvey} we deduce that (this conjecture implies that) all Chow motives are finite dimensional.

Hence this conjecture implies Theorem \ref{tnum} below and also an improved version of Theorem \ref{tzeta} (since Chow motives strongly generate $\dmgm$; see Lemma 2.3.1(1) of \cite{bzp} and Proposition 5.2.2 
of \cite{bws}). It also has several other nice consequences (see \S7 of  \cite{kim}). However, the author does not know how to deduce
this conjecture 
from Assumption \ref{assmain}. 

4. Recall also that Theorem 3 of \cite{murre} gives a Chow-K\"unneth decomposition of the motif of any smooth projective (connected) surface. Thus our Assumption \ref{assmain} implies that motives of smooth projective surfaces are Kimura-finite (and in the case $p=0$ this statement also follows from Theorem I of \cite{ayoubcon}; cf. Remark \ref{rinfield}(1) below). Hence all the objects of the tensor additive idempotent complete subcategory of $\chow$ generated by motives of surfaces and Tate twists are finite dimensional as well (see Lemma 3.7(3) of \cite{andsurvey}). 
 Certainly, this conclusion is much better than its "one-dimensional" analogue (cf. Corollary 4.4 of \cite{kim} and Corollary 4 of \cite{gul}). Note however that this observation is certainly not sufficient to deduce the main results of the current paper. 

5. It is well-known that all objects of $D^b(\ql-\vecto)$ are {\it Schur-finite} in this category in the sense of Definition 2.1 of \cite{mazza} (see Example 2.2 or Corollary 2.27 of ibid.). Thus Assumption \ref{assmain} immediately implies that all objects of $\dmgm$ are 
Schur-finite as well (cf. the proof of Theorem \ref{tfdim}(1)). 
However, the (general case of the) notion of Schur-finiteness does not appear to be really useful for applications. 

\end{rema}

\section{On the rationality of 
 motivic zeta functions}\label{szeta}

Now we pass to motivic zeta functions; these are defined as series with coefficients in certain 
 $K_0$-rings of motives.
So we recall the definitions of Grothendieck rings of  Voevodsky and Chow motives (cf. Remark \ref{rkz}) 
 and of the corresponding zeta functions. 

\begin{defi}\label{dkz}
1. $K_0(\dmgm)$ is the Abelian group 
 whose generators are $\{[M]$, $M\in  \obj \dmgm\}$, and the relations are of the form $[B]=[A]+[C]$ if
$A\to B\to C\to A[1]$ is a distinguished triangle. 

2. 
We define $K_0(\chow)$  as the Abelian group 
 whose generators are $\{[M]$, $M\in  \obj\chow\}$,
and the relations are $[B]=[A]+[C]$ if
$A,B,C\in \obj\chow$ and $B\cong A\bigoplus C$.

3. For $M\in\obj \dmgm$ we define
$\zeta(M)=1+ \sum_{i> 0}[\operatorname{Sym}^i(M)]T^i\in
K_{0}(\dmgm)[[T]]$.
\end{defi}

\begin{rema}\label{rkz}
1. Let us recall that there exists a full embedding of $\chow$ into $\dmgm$ that respects tensor products. Its essential image is the class of retracts of $\cup_{j\in \z}\mg(\spv)\lan j\ra$. Thus for any smooth projective $k$-variety $P$ and $r>0$ one may assume that the objects $\wedge^{r}{\mg(P)}$ and $\operatorname{Sym}^r(\mg(P))$ belong to $\obj \chow$.


2. 
Next, the embedding $\chow\to \dmgm$  induces an isomorphism $K_0(\chow)\cong K_0(\dmgm)$; see Proposition 2.3.3 of \cite{bzp} (and also Corollary 6.4.3 of \cite{mymot} for the case $p=0$).

So we will put  the motivic zeta functions of Definition \ref{dkz}(3) into $K_{0}(\chow)[[T]]$.
We will say more on this matter 
 in Remarks  \ref{rzeta}(1) and \ref{rzetavar}(1) below.
\end{rema}

\begin{theo}\label{tzeta}

1. If $A\to B\to C\to A[1]$ is a distinguished triangle in $\dmgm$ then $\zeta(B)=\zeta(A)\cdot \zeta(C)$.

2. Denote by $\dms$ the 
 triangulated subcategory of $\dmgm$ 
 strongly generated (see \S\ref{smot}) by $\cup_{j\in \z}\mg(\sv)\lan j\ra$. 
 Then for any object $M$ of $ \dms$ the zeta function  $\zeta(M)$ is a rational series, i.e., it is of the form $f(T)/g(T)$ for some $f,g\in 1+ TK_{0}(\chow)[T]$.

\end{theo}
\begin{proof}
1. This is Corollary 3 of \cite{gul}.

2. According to 
 assertion 1, it suffices to find a set of objects $S\subset \obj \dms$ that strongly generates $\dms$ (as a subcategory of $\dmgm$)
 and such that $\zeta(M)$ is rational (in the sense described above) for any $M\in S$.

 Now we take $S$ consisting of finite dimensional motives of the form $\co(\mg(Z\to A))\lan i \ra$ as given by Theorem \ref{tfdim}(2). If a motif $M$ is 
 finite-dimensional then its zeta function series is rational (see Proposition 4.6 of \cite{andsurvey}).\footnote{Actually, loc. cit. is formulated for "pure" motives; yet the statement relies on the well-known Formulaire 3.2(5) of ibid., and it yields that $\zeta(M)$ is rational even if we consider the finer Grothendieck ring of $\dmgm$ where we assume that $[B]=[A]+[C]$ only if $A,B,C\in \obj\dmgm$ and $B\cong A\bigoplus C$.
}
Hence it remains to verify that any motif of the form $\mg(U)\lan j\ra$ (for a smooth $k$-variety $U$ and  $j\in \z$) belongs to the subcategory $\dmsp$ of $\dmgm$ strongly generated by $S$. 

Next, we can certainly assume $j=0$ and prove the latter statement by induction on the dimension of $U$. We also recall that for any 
 open cover $W=V\cup Y$ of a smooth $k$-variety $W$ we have (essentially, by the definition of $\dmgm$) the Mayer-Vietoris distinguished triangle $$\mg(U\cap V)\to \mg(V)\bigoplus \mg(Y)\to \mg(W)\to \mg(U\cap V)[1].$$ Hence we can assume that $U$ is connected and affine. Choosing the corresponding generic hyperplane section $T$ of $U$ (recall that $k$ is infinite) we deduce from the inductive assumption that both $\mg(T)$ and $\mg(T\to U)$ are objects of $\dmsp$; hence  $\mg(U)$ also is.

\end{proof}

\begin{rema}\label{rzeta}
1. In particular, $\zeta(M)$ is a rational series if $M=\mg(P)$, where $P$ is a smooth projective $k$-variety. Now, in this case it is quite "natural" to call $\zeta(M)$ the motivic zeta function of $P$; note also that this zeta function can be computed without using the isomorphism $K_0(\dmgm)\cong K_0(\chow)$ (i.e., using part 2 of Definition \ref{dkz}).  However, 
we prefer to define motivic zeta functions of general varieties via motives with compact support; see 
 \S\ref{sgs} below.

2. Applying the well-known motivic Gysin distinguished triangle one easily obtains that $\dms$ coincides with  the triangulated subcategory of $\dmgm$ strongly generated by $\cup_{j\in \z}\mg(\spv)\lan j \ra$
	 whenever $k$ admits resolution of singularities (in particular, if $p=0$). 
	Still it is not known whether this statement is fulfilled if $p>0$. Moreover, below we will describe a subcategory of $\dmgm$ that is a priori bigger than $\dms$ and satisfies the rationality of zeta functions property; this subcategory contains all motives with compact support of varieties.
	
	
	3. We would certainly like to prove that zeta functions are rational for arbitrary Chow (and thus also for Voevodsky) 
	 motives (even though 
	 Proposition \ref{pzetavar}(II.2) below gives the rationality of zeta functions of all "reasonable" motives). 
	 The author does not know how to prove that 
	$\zeta(N)$ is rational whenever $\zeta(M)$ is and $N$ is retract of $M$. The corresponding rationality implication is obviously wrong for a general triangulated category (since any $N$ is a 
	 retract of $N\bigoplus N[1]$; cf. Definition \ref{dkz}(1)), and it could be wrong for "additive versions" of zeta functions (cf. part 2 of the definition). 
	
	Possibly, one can deduce the rationality statement in question from the Krull-Schmidt theorem for Chow motives; moreover, it obviously follows from the aforementioned Conjecture 7.1 of \cite{kim} (see Remark \ref{rnori}(3)).
	


\end{rema}

\section{Applications to numerical motives}\label{skim}

\begin{theo}\label{tnum}
If Assumption \ref{assmain} is fulfilled then all numerical motives (over $k$) are Kimura finite dimensional.

\end{theo}
\begin{proof}
We will use the notation $\motn$ for the category of numerical motives.
It certainly suffices to verify that all numerical motives are Kimura-finite as objects of the category $K^b(\motn)\supset \motn$.

We recall the existence of an exact tensor {\it weight complex} functor $\dmgm\to K^b(\chow)$ whose composition with the aforementioned 
 embedding $\chow\to \dmgm$ is the obvious embedding $\chow\to K^b(\chow)$; see Lemma 20 of \cite{bachinv} along with the text after it (cf. also Proposition 2.3.2(1) of \cite{bzp} and \S6.3 of \cite{bws}). Next we compose this functor with the obvious functor $K^b(\chow)\to K^b(\motn)$; 
 we will write $\tnum$ for this composition.  Certainly, $\tnum$ is an exact tensor functor as well; hence it sends finite-dimensional objects of $\dmgm$ into that of $K^b(\motn)$.

Next, we have demonstrated in the proof of Theorem \ref{tzeta}(2) that the category $\dms$ is strongly generated by finite dimensional objects of the form $\co(\mg(Z\to A))\lan i\ra$ (see Theorem \ref{tfdim}(2)). Thus for any smooth projective $P$ its numerical motive (considered as an object of $K^b(\motn)$) can be obtained from the finite dimensional motives of the form $\tnum(\co(\mg(Z\to A))\lan i\ra)$ by means of shifts and taking cones of morphisms.  

Now, $\motn$ is an abelian semisimple category according to Theorem 1 of \cite{ja}. 
 Thus if $A\to B\to C\to A[1]$ is a distinguished triangle in $K^b(\motn)$ then $B$ is a direct summand of $A\bigoplus C$ 
(look at the long exact sequence for the $\motn$-homology of these complexes). Since direct sums and retracts of Kimura-finite objects  are Kimura-finite by Lemma 3.7(3) of \cite{andsurvey}, and finite dimensionality is also respected by shifts (cf. Theorem 1 of \cite{gul}),  we obtain that numerical motives of smooth projective varieties are Kimura-finite as well. Hence arbitrary numerical motives are finite dimensional also (apply 
Lemma 3.7(3) of \cite{andsurvey} once again).

\end{proof}

\begin{coro}
The category $\motn$ is Tannakian (if we modify the commutativity constraint for the tensor product on it).

\end{coro}
\begin{proof}

The claim follows from Theorem \ref{tnum} immediately according to Theorem 9.2.2 of \cite{aka}. 

\end{proof}


\section{Supplements: on zeta functions of varieties and conservativity of realizations}\label{suppl}

In \S\ref{sgs} we define motivic zeta functions of arbitrary varieties as the ones of their motives with compact support, and establish their rationality.

In \S\ref{sreal} we recall that in the case $p=0$ the conservativity of \'etale realizations is equivalent to that of the De Rham one; this relates our Assumption \S\ref{assmain} to the results and announcements of \cite{ayoubcon}. 

\subsection{On motivic zeta functions of varieties}\label{sgs}

We recall some properties of motives with compact support and study the corresponding zeta functions. Note that in the case $p=0$ part II.2 of the following proposition can be easily deduced from Theorem \ref{tzeta}(2) directly (see Remark \ref{rzeta}(2)).

\begin{pr}\label{pzetavar}
I. Motives with compact support $\mgc(-)$ of  $k$-varieties (as provided by \S4.1 of \cite{1} along with \S5.3 of \cite{kellyast}) enjoy the following properties.

\begin{enumerate}
\item\label{imceq}
We have $\mgc(X)=\mg(X)$  whenever $X\in \spv$. 
Moreover, $\mgc(X)$ 
 is an object $\dmgm$ for any $X\in \var$.

\item\label{imctr} If $Z$ is a closed subvariety of a $k$-variety $X$ 
 and $U=X\setminus Z$ then there exists a distinguished triangle 
$\mgc(Z) 
 {\to} \mgc(X)\to \mgc(U)\to \mgc(Z)[1].$ 

\item\label{imcd} If $X$ is a smooth $k$-variety and (all its components) are of dimension $d\ge 0$ then $
{\mg(X)}^{\widehat{\ \ \ }}\cong \mgc(X)\lan -d\ra$, where $
{\mg(X)}^{\widehat{\ \ \ }}$ denotes the dual to ${\mg(X)}$ in the (tensor) category $\dmgm$.

\end{enumerate}

II. For $X\in \var$ define $\zeta(X)\in K_0(\chow)[[T]]$ as $\zeta(\mgc(X))$. Then the following statements are valid.

1. For $Z,X,$ and $U$ as in assertion I.\ref{imctr} we have $\zeta(X)=\zeta(Z)\cdot \zeta(U)$.

2. Let $\dmsc$ denote the smallest strictly full tensor triangulated subcategory of $\dmgm$ that contains  both $\dms$ (see Theorem \ref{tzeta}(2)) and $\mgc(X)$ for any $X\in \var$; assume that  Assumption \ref{assmain} is fulfilled (for our $k$). Then for any $M\in \obj \dmsc$ the zeta function  $\zeta(M)$ is a rational series. In particular, all zeta functions of varieties are rational. 

\end{pr}
\begin{proof}
I. These statements easily follow from the results of \cite{kellyast}; see  Theorem 5.3.18 of ibid. and Proposition 4.1.1(1,3) of \cite{bscwh}.

II.1. This is immediate from the join of assertion I.\ref{imctr} with Theorem \ref{tzeta}(1).  

2. Applying assertion I.\ref{imctr} we obtain that $\dmsc$ is the smallest  strict tensor triangulated subcategory of $\dmgm$ that contains both $\dms$ and $\mgc(\sv)$. 
 Combining the duality statement provided by assertion I.\ref{imcd}  with the argument used in the proof of Theorem \ref{tzeta}(2) we obtain that $\dmsc$ is generated in this sense by motives of the form $\co(\mg(Z\to A))\lan j \ra$ and 
 ${\co(\mg(Z\to A))}^{\widehat{\ \ \ }}$ for $Z\to A$ as in Theorem \ref{tfdim}(2). Thus  $\dmsc$ also equals the 
 triangulated subcategory of $\dmgm$ 
 strongly generated by tensor products of the form $$\otimes_{1\le i\le a} \co(\mg(Z_i\to A_i)) \otimes (\otimes_{1\le m\le b} (\co(\mg(Z_m\to A_m))))^{\widehat{\ \ \ }}\lan j \ra$$ for $Z_i\to A_i$ and $Z_m\to A_m$ of this sort and $j\in \z$. Now, tensor products of this form are finite-dimensional since all their multipliers are (see  Theorem \ref{tfdim}(2) and Lemma 3.7(3) of \cite{andsurvey}); thus it remains to apply Proposition 4.6 of ibid. (cf. the proof of Theorem \ref{tzeta}(2)). 

\end{proof}

\begin{rema}\label{rzetavar}
1. There are several reasons to call $\zeta(X)$ (see part II.1 of our proposition) the motivic zeta function of a variety $X$. Firstly, the relation given by our assertion II.1 is what one awaits from zeta functions of varieties. 
 
Secondly, it is well-known that  the number of points of a variety $X$ over a finite field (of characteristic $p$) is determined by the $\ql$-adic \'etale  cohomology of $X$ with compact support, whereas the latter is 
 functorial in $\mgc(X)$ (see Proposition 4.2.5(1) of \cite{bscwh}); this relates $\zeta(X)$ with Hasse-Weil zeta functions of varieties.

Thirdly, for any 
$M\in \obj \dmgm$ 
 its class in $K_0(\chow)$ can be computed as the class of $t(M)$ in $K_0(K^b(\chow))\cong K_0(\chow)$ (where $K^b(\chow)$ is defined similarly to Definition \ref{dkz}(1) and $t:\dmgm\to K^b(\chow)$ is the exact weight complex functor mentioned in the proof of Theorem \ref{tnum}); see the proof of \cite[Theorem 5.3.1]{bws} or of \cite[Theorem 6.4.2]{mymot}. Next, if $p=0$ then for $M=\mgc(X)$ (where $X$ is a $k$-variety) the object $t(M)$ is isomorphic to the  weight complex of $X$ provided by Theorem 2 of \cite{gs} (see Proposition 6.6.2 of \cite{mymot}).\footnote{Recall that the term "weight complex"  was originally introduced in \cite{gs}; yet the weight complex theory developed in \S3 and \S6 of \cite{bws} is 
much more general and powerful.} Hence $\zeta(X)$ coincides with the series $\zeta_{M(X)}$ 
 defined in \S2 of \cite{gul}. 

2. In the case $p=0$ the category $\dmsc$ is easily seen to coincide with the category $\dms$ described in Theorem \ref{tzeta}(2) (cf. Remark \ref{rzeta}(1)). Since this case is the main one for the current paper, the "extension" of $\dms$ given by part II.2 of our proposition does not appear to be really important. 
\end{rema}

\subsection{On the conservativity of various realizations}\label{sreal}

Let us discuss the conservativity of various realization.

\begin{pr}\label{pcons}
Assume that $p=\cha k=0$. Then Assumption \ref{assmain} is fulfilled if and only if the naturally defined De Rham realization $RH_{dR}:\dmgmop\to D^b(k-\vecto)$ (cf. \cite{ayoubcon}) is conservative on $\dmgm$.

\end{pr}
\begin{proof}

Assume that $k$ is a subfield of the field $\com$ of complex numbers. Then the existence of the "mixed" realization of $k$-motives provided by Theorem 2.3.3 of \cite{hureal} gives the result immediately (see Definitions 2.2.1 and 2.2.2 of ibid.; cf. also \cite{hucorr}). 

The case of a general (characteristic $0$) field $k$ reduces to this one via a standard reasoning. For $M\in \obj \dmgm$ we should prove that 
 $\rhetl(M)=0$ if and only if  $RH_{dR}(M)=0$. Now, let $k_0\subset k$ be a finitely generated field of definition for $M$ (see \S2.3 of \cite{bscwh}), i.e., $M$ can be obtained from some motif $M_0\in \obj \dmgm(k_0)$ via base field change. Then we 
  have $RH_{dR}(M)\cong RH_{dR}(M_0)\otimes_{k_0}k$ since we have a similar canonical isomorphism of this sort for the cohomology of smooth varieties (here one can apply Theorem B.2.2 of \cite{hucorr}). 
A similar argument yields	
 that the \'etale realizations of $M_0$ and $M$ are canonically isomorphic; 
 thus we can assume that $k_0\subset \com$ indeed.
\end{proof}

\begin{rema}\label{rinfield}
1. 
The author was motivated to write this paper by 
 the recent proof of the conservativity of the restriction of $RH_{dR}$ to the category $\chow$ in the case $p=0$;  
 see Theorem I of \cite{ayoubcon} (unfortunately, the current version of this proof contains a gap). 
 Certainly, this statement implies the seminal Bloch conjecture (and this implication is mentioned in 
 ibid. as well). 
 It is also noted 
 that one can deduce the conservativity of $RH_{dR}$ on $\dmgm$ (for $p=0$) from Theorem II of ibid.; 
 cf. Theorem 1.5.1 and Remark 3.1.4(3) of \cite{bwcomp}. 

Thus the current paper can be thought of as of a certain "complement" to \cite{ayoubcon}. 

2. Probably the assumption that $k$ is infinite is not really important for the main statements of the current paper. Assumption \ref{assmain} for a field $k$ appears to be equivalent to that for its algebraic closure (since one can consider fields of definitions of $\kalg$-motives that are finite over $k$ and restrict scalars; cf. the proof of Proposition \ref{pcons}). Since the base field change functor $\motn(k)\to \motn(\kalg)$ is well-known to be faithful, this chain of arguments gives Theorem \ref{tnum} 
in the case of a finite $k$.

The author suspects that one can also extend Theorem \ref{tzeta} to the case of finite $k$ using somewhat similar arguments; yet he did not think much about this question since  currently the case $p>0$ is not really  actual. 


\end{rema}

\end{document}